\documentclass[reqno, a4paper]{amsart}
\usepackage{amssymb, amscd, amsfonts,  amsthm}

\usepackage[mathscr]{eucal}
\usepackage{graphicx}
\usepackage{a4wide}
\usepackage[all,poly]{xy} 
\usepackage[all]{xy} 

\newtheorem{theorem}{Theorem}[section]
\newtheorem{lemma}[theorem]{Lemma}
\newtheorem{corollary}[theorem]{Corollary}
\newtheorem{proposition}[theorem]{Proposition}
\newtheorem{example}[theorem]{Example}

\newtheorem{problem}[theorem]{Problem}

\theoremstyle{definition}

\theoremstyle{definition}
\newtheorem{definition}[theorem]{Definition}
\newtheorem{remark}[theorem]{Remark}

\newcommand{\restrict}{\,{\mathbin{\vert\mkern-0.3mu\grave{}}}\,}

\newcommand{\luk}{\L u\-ka\-s\-ie\-wicz}

\newcommand{\remove}[1]{}

\DeclareMathOperator{\McNn}{\mathcal M([0,1]^{\it n})}

\DeclareMathOperator{\McN}{\mathcal M}
\DeclareMathOperator{\conv}{\rm conv}
\DeclareMathOperator{\den}{\rm den}

\DeclareMathOperator{\I}{[0,1]}
\DeclareMathOperator{\cube}{[0,1]^{\it n}}
\DeclareMathOperator{\kube}{[0,1]^{\it k}}

 \title[MV-algebraic Stone-Weierstrass]
{A Stone-Weierstrass 
theorem 
for  MV-algebras and unital $\ell$-groups}

\author{\sc Leonardo Manuel Cabrer and Daniele Mundici}

\address[L.M. Cabrer]{Department of Statistics,
Computer Science and Applications,  ``Giu\-sep\-pe Parenti''\\ 
University of Florence\\
Viale Morgagni 59 
50134\\ Florence \\
Italy}
\email{ l.cabrer@disia.unifi.it }

\address[D. Mundici]{Department of
Mathematics and Computer Science  ``Ulisse Dini'' \\
University of Florence\\
Viale Morgagni 67/A \\
I-50134 Florence \\
Italy}
\email{ mundici@math.unifi.it }

\begin{document}

\thanks{2010 {\it Mathematics Subject Classification.}
Primary:  06D35.    Secondary:   03B50, 03D40,
  05E45, 06F20, 08A50, 08B30,  
  52B11, 52B20,  54C15,     55U10,
57Q05, 57Q25}

\keywords{MV-algebra, lattice ordered abelian group,
strong order  unit,  unital $\ell$-group, projective MV-algebra, 
 infinite-valued \luk\ calculus 
Stone-Weierstrass theorem, retraction,  McNaughton function, 
piecewise linear function,  basis,  Schauder hat,
regular triangulation, unimodular triangulation, duality, polyhedron,
  finite presentation, isomorphism problem, Markov unrecognizability
  theorem.}

\begin{abstract}
Working jointly in the equivalent categories of MV-al\-ge\-bras and
lattice-ordered abelian groups with strong order unit 
 (for short, unital $\ell$-groups), 
we prove that  isomorphism is a sufficient condition for
a separating subalgebra $A$ of a finitely presented  algebra $F$
to coincide with $F$.    
The separation and isomorphism conditions do not
individually  imply
$A=F$.   Various related problems,
like the separation property of $A$, 
 or  $A\cong F$  (for   $A$ a separating subalgebra of $F$),
 are   shown to be (Turing-)decidable.  
We use tools from algebraic topology, category theory, 
polyhedral geometry and computational algebraic logic.
\end{abstract}

\maketitle



\section{Introduction}
A {\it unital $\ell$-group}, \cite{bigkeiwol, gla}
group equipped with a translation invariant
 lattice structure, and $u\geq 0$ 
  is an element whose positive integer 
  multiples eventually dominate
every element of $G$.
An {\it MV-algebra} $A=(A,0,\oplus,\neg)$ is an abelian monoid $(A,0,\oplus)$
equipped
with an operation $\neg$
such that   $\neg\neg x=x, \,\,\,x\oplus \neg 0=\neg 0$
and $\neg(\neg x \oplus y )\oplus y=\neg(\neg y \oplus x )\oplus x$.  
Recently, MV-algebras, and especially finitely presented MV-algebras, 
\cite{marspa},   \cite[\S 6]{mun11}, 
have   found  applications  to
lattice-ordered abelian groups 
\cite{cab, cabmun-ja, cabmun-ccm, mmm, mun08},  
  the Farey-Stern-Brocot  AF C$^*$-algebra
of   \cite{mun88},  (see \cite{boc,eck,mun-lincei}),
probability and measure theory,
\cite{forum,mun-cpc}, 
multisets \cite{cigmar},
and    vector lattices \cite{ped}.  
The versatility of
 MV-algebras stems from a number of factors, including:
\begin{itemize}
\item[(i)] The categorical
equivalence  $\Gamma$  between  unital 
$\ell$-groups and MV-algebras  \cite{mun86},  which endows 
unital $\ell$-groups with the equational machinery of free algebras,
finite presentability, and word problems,
 despite the archimedean property of
the unit is not definable by equations.
\item[(ii)]  The duality
between finitely presented MV-algebras and rational polyhedra, \cite{cab, marspa, mun11}.
\item[(iii)]   The one-to-one correspondence,  via  $\Gamma$
and Grothendieck's  $K_0$,  between countable 
MV-algebras and   AF C$^*$-algebras 
whose Murray-von Neumann order of
projections is a lattice,  \cite{mun86}.   
\item[(iv)]   The deductive-algorithmic machinery of  the
 infinite-valued \luk\ calculus  \L$_\infty$  is 
 immediately applicable to MV-algebras,
 \cite{cigdotmun,mun11}.
 \end{itemize}

This paper deals with
 finitely generated subalgebras
of finitely presented MV-algebras and  unital $\ell$-groups.
We will preferably  focus on the  equational class of
MV-algebras, where freeness and finite presentations
are immediately definable.
Finitely presented MV-algebras are the Lindenbaum 
algebras of finitely axiomatizable theories in
 \L$_\infty$. Finitely generated projective
 MV-algebras, in particular,  are a key tool for
 the proof-theory of    \L$_\infty$ (see \cite{cab,  jer, marspa}
 for an algebraic analysis of admissibility, exactness and unification
 in   \L$_\infty$).
 Remarkably enough,  the characterization of 
 projective  MV-algebras and unital $\ell$-groups 
 is a deep {\it open} problem in algebraic topology, 
 showing that  unital $\ell$-groups have
 a greater complexity than  $\ell$-groups,  
(see \cite{cabmun-ja, cabmun-ccm}). 
 As a matter of fact, while
  the well-known  Baker-Beynon duality
  (\cite{bey} and references therein), 
shows  that finitely presented $\ell$-groups
coincide with finitely generated projective $\ell$-groups,
the class of  finitely generated projective   unital   $\ell$-groups
(resp., 
finitely generated projective
MV-algebras) is strictly contained in the class of 
 finitely presented unital   $\ell$-groups (resp., 
  finitely presented MV-algebras).

Let $A$ be a finitely generated subalgebra  of a 
finitely presented MV-algebra,
(or  unital $\ell$-group)   $F$.
In Theorem \ref{theorem:separation-is-decidable} 
we prove: it is decidable whether $A$ is separating. 
If  $A$ is separating  and $F$ is a finitely generated free MV-algebra
then $A$ is projective. This is proved in Theorem \ref{theorem:projective}.
Further,  $A=F$ iff $A\cong F$: this is our MV-algebraic
 Stone Weierstrass theorem 
(\ref{theorem:sw}). 
For separating subalgebras $A$
of free $n$-generator  MV-algebras
or unital $\ell$-groups,
  the isomorphism problem
$A\cong F$
  is decidable. See Theorem \ref{theorem:decidable-sw}.


 As is well known,  \cite[9.1]{cigdotmun}, the free $n$-generator
MV-algebra $\McNn$ consists of all
$n$-variable McNaughton functions defined
on $\I^n$. For any MV-term $\tau=\tau(X_1,\ldots,X_n)$
we let $\hat\tau\in \McNn$ be obtained
 by evaluating $\tau$ in $\McNn$.
In Section \ref{section:method} 
we introduce a method  to write down a list
of MV-terms  $\tau_1,\ldots,\tau_k$  in the variables
$X_1,\ldots,X_n$ in such a way that  the subalgebra
 of  $\McNn$  
generated by 
$\hat\tau_1,\ldots,\hat\tau_k$  is
  {\it separating and  distinct} from $\McNn$.
In Theorem   \ref{theorem:delta-basis} we prove that every  
  finitely generated separating
proper subalgebra $A$ of $\McNn$ is obtainable by this method.
In the light of Theorem \ref{theorem:projective},   a large class of
projective non-free  MV-algebras can be effectively introduced.
%
%
%
%
%
In  Section \ref{section:basis}, we connect our presentation of 
finitely generated separating subalgebra $A$ of $\McNn$ with the 
notion of basis \cite[\S 6]{mun11}. In Theorem   \ref{theorem:effective},
for any finitely generated separating subalgebra $A$ of $\McNn$  
we provide an effective method to transform 
every generating set of $A$ into a 
basis 
of $A$. 
 As  another application,
 in  Section \ref{section:recognizing},
we prove the decidability
 of the problem of recognizing whether two
different sets of terms generate the same 
separating subalgebra of 
a free   MV-algebra.

Finally, in
 Section \ref{section:final} we discuss the mutual relations between
presentations  of MV-algebras as {\it finitely generated
subalgebras}
of free MV-algebras, 
and the traditional finite presentations in terms of {\it principal
quotients} of free MV-algebras.
Many results proved in this paper for separating
   finitely generated  unital $\ell$-subgroups of free
unital $\ell$-groups,  fail for
finitely presented $\ell$-groups, and are open problems for 
  finitely presented  unital $\ell$-groups.
 The separation hypothesis   plays a crucial role in
most decidability results of the earlier sections.
Actually, the final two results of this paper
(Theorems  \ref{theorem:generators-to-quotient}
and \ref{theorem:finale}) show that
without  this hypothesis,
  decision problems for a 
subalgebra $A$ of $\McNn$ generated by
$\hat\tau_1,\ldots,\hat\tau_k$ become  
as difficult as their classical counterparts where $A$
is  
presented  as a principal
quotient of $\McNn$. 
%
%
%
\section{MV-algebras, rational polyhedra, regular triangulations}
\label{section:presentation}
\subsubsection*{MV-algebras, \cite{cigdotmun,mun11}}
We assume familiarity with the categorical equivalence  $\Gamma$ 
between MV-algebras and unital $\ell$-groups, \cite{mun86,cigdotmun}.
%
For any  closed  set $Y \subseteq \cube$ we let
$\McN(Y)$ denote the MV-algebra of restrictions to $Y$
of all McNaughton functions defined on $\cube$.
A set  $S\subseteq \McN(Y)$
is said to be {\it separating}  (or,  $S$ {\it separates points of $Y$})
   if for all $x,y\in Y$ such that $x\neq y$ 
there is $f\in S$
with $f(x)\not=f(y)$.
Unless otherwise specified,  $Y$ will be nonempty,
whence  the MV-algebra $ \McN(Y)$ will be  nontrivial.

By an  {\it MV-term}  
$\tau = \tau(X_{1},\ldots,X_{n})$ 
we mean  a string of symbols obtained from
the variable symbols $X_{i}$ and the constant symbol $0$ by
a finite number of applications of the MV-algebraic connectives
$\neg,\oplus$.  
The map $\,\,\,\hat{}\,\,\,$  sending $X_{i}$ to the $i$th coordinate
function $\pi_{i}\colon [0,1]^{n}\to [0,1]$  
canonically extends to a map interpreting
each MV-term $\tau(X_1,\ldots,X_n)$ as a McNaughton
 function $\hat{\tau}\in \McNn$.  
 McNaughton theorem 
 \cite[9.1.5]{cigdotmun} 
 states that this map is onto the
 MV-algebra $\McNn$.
The set $\pi_1,\ldots,\pi_n$
freely generates  the free MV-algebra $\McNn$. 
 
\subsubsection*{Rational polyhedra 
and their regular triangulations, \cite{mun11, sta}}
Let  $n=1,2,\ldots$ be a fixed integer.
A point  $x$  lying in the $n$-cube  $\cube$ is said to be
 {\it rational} if
so are its coordinates.   In this case,  there are uniquely
determined integers  $0\leq c_i \leq d_i$  such that 
 $x$ can be written as 
$
x=(c_1/d_1,\ldots,c_n/d_n),\,$ with $  c_i \,\,{\rm and} \,\,
d_i$    relatively prime for each  $i=1,\ldots,n$.
By definition, the
 {\it homogeneous correspondent} of $x$ is the integer
vector
$
\tilde x = (dc_1/d_1,\ldots,dc_n/d_n,d)  \in \mathbb Z^{n+1},
$
where $d>0$ is the least common multiple of $d_1,\ldots,d_n$.
The integer $d$ is said to be the {\it denominator of }  $x$,
denoted  $\den(x)$.
A {\it rational polyhedron  $P$ in $\cube$}  is a 
subset of  $\cube$  which is a finite union of simplexes in $\cube$
with rational vertices.  
A polyhedral complex $\Pi$ is said to be {\it rational}  if the
vertices of each polyhedron in $\Pi$ are rational.  

A rational $m$-dimensional
simplex  $T=\conv(v_0,\ldots,v_m)\subseteq \cube$ 
 is {\it regular} if the set   $\{\tilde v_0, . . . , \tilde v_m\}$ of homogeneous
 correspondents   of its vertices  can be extended to a  basis of
 the free abelian group
 $\mathbb Z^{n+1}$.  An (always finite)  simplicial complex
 $\Delta$  is
 {\it regular}  if each one of its simplexes is {\it regular}.
 The {\it support}  $|\Delta|$ of  $\Delta$, i.e., the point-set union of all
 simplexes of $\Delta$, is the most general possible
 rational polyhedron in $\mathbb R^n$  (see \cite[2.10]{mun11}).
 We also say that $\Delta$ is a regular {\it triangulation} of the
 rational polyhedron $|\Delta|$.

 Regular simplexes and complexes are called ``unimodular''
 in \cite{cigdotmun, mmm, marspa, mun08}, and ``(Farey) regular''
 in \cite{mun-cpc}.
   In the literature on polyhedral topology, 
 notably in \cite{sta}, the adjective ``regular''
 has a different meaning.
 Throughout this paper, the adjective ``linear'' is to be understood in the
 affine sense.

\subsubsection*{Finitely presented MV-algebras are dual to
rational polyhedra, \cite{cab,marspa, mun11}} 
An MV-algebra $A$ is said to be {\it finitely presented} if it
 is isomorphic to the quotient MV-algebra
  $\McNn/\mathfrak j$, for some $n=1,2,\ldots,$ 
and some principal ideal $\mathfrak j$ of $\McNn$. 
If $\mathfrak j$ is generated by $g\in \McNn$
 then $\McNn/\mathfrak j\cong\McN(g^{-1}(0))$,
 and $g^{-1}(0)$ is a  (possibly empty) 
  rational polyhedron in $\I^n$.

Given two rational polyhedra $P\subseteq \cube$ and  $Q\subseteq \kube$ a {\it $\mathbb{Z}$-map}
 is a continuous piecewise linear map $\eta\colon P \to Q$  
    such that each linear piece of $\eta$ has
    integer coefficients.    
    Following  \cite[3.2]{mun11},  given
rational polyhedra  $P\subseteq \cube$  and $Q\subseteq \kube$,
we write  $P\cong_\mathbb Z Q$ (and say that $P$ and $Q$
are {\it $\mathbb Z$-homeomorphic}) if there is a homeomorphism
$h$  of $P$ onto $Q$ such that both $h$ and $h^{-1}$ 
are $\mathbb Z$-maps.  
We also say that $h$ is a
{\it $\mathbb Z$-homeomorphism}.

The functor $\McN$  sending
 each polyhedron $P$ to the MV-algebra
 $\McN(P)$,   and each  $\mathbb{Z}$-map $\eta\colon P\to Q$ to the map $\McN(\eta)\colon \McN(Q)\to \McN(P)$ defined by
 \begin{equation}
 \label{equation:functor} 
 \McN(\eta)\colon f\mapsto  f\circ \eta,\,\,\,\mbox{for any }f\in\McN(Q),
\,\, \mbox{\rm where $\circ$ denotes composition,}
 \end{equation}
  determines a categorical equivalence between
    rational polyhedra with $\mathbb Z$-maps, and
    the opposite of the category of  
 finitely presented MV-algebras with homomorphisms.
 For short,  $\McN$ is a {\it duality} between these categories.
 See \cite{cab, marspa, mun11} for further details.

\section{``Presenting'' MV-algebras by a finite list of MV-terms}
Every finite set of MV-terms
$\tau_1,\ldots,\tau_k$ in the variables $X_1,\ldots,X_n$
 determines the subalgebra $A$ of $\McNn$
generated by the McNaughton functions
$\hat\tau_1,\ldots,\hat\tau_k$. 
Then $A$ is  finitely presented \cite[6.6]{mun11}.
As we will show throughout this paper, when  $A$ is 
presented  via generators 
$\hat\tau_1,\ldots,\hat\tau_k,\,\,$   several decision problems
turn out to be solvable.
 
 \smallskip
 Our first example is as follows:   
  
  \begin{theorem}
  \label{theorem:separation-is-decidable}
The following {\rm  separation problem}  
is decidable:

\smallskip
\noindent ${\mathsf{INSTANCE}:}$   A list of MV-terms 
$\tau_1,\ldots,\tau_k$ in the variables $X_1,\ldots,X_n$.

 \smallskip
\noindent ${\mathsf{QUESTION}:}$ Does the set  of  functions 
$\{\hat\tau_1,\ldots,\hat\tau_k\}$ separate points of $\cube$ ? 
  \end{theorem}

  \begin{proof} 
%
Let
  $g
=(\hat\tau_1,\ldots,\hat\tau_k)\colon\cube\to\kube$
  be defined by $g(x)=(\hat\tau_1(x),\dots,\hat\tau_k(x))$ for all
  $x \in \cube$.

By \cite[3.4]{mun11}, the range  $R$ of $g$ is
a rational polyhedron  in $\kube$.
The separation
problem equivalently asks if $g$ is one-to-one.
Equivalently, is $g$ is a 
piecewise linear homeomorphism of $\cube$  onto $R$?
Fix $i\in \{1,\dots,k\}$. By induction on the number of connectives
in the subterms of $\tau_i$ one can effectively list the linear pieces
$l_{i1},\dots,l_{it_i}$ of the piecewise linear function
$\hat\tau_i$.
 Since all these linear pieces   have integer coefficients,
the routine stratification argument of 
\cite[2.1]{mun11} yields
 (as the list of the sets of vertices of its simplexes)
 a rational polyhedral complex  $\Pi_i$  over $\cube$ such that
$\hat\tau_i$
 is linear on each simplex of $\Pi_i$.
As in  \cite[1.4]{sta}, we now subdivide
$\Pi_i$   into a rational triangulation
$\Delta_i$ without adding new vertices.   
By induction on  $k,$ we compute
 a rational  triangulation
$\Delta$ of $\cube$  which is a joint subdivision
of  $\Delta_1,\dots,\Delta_k$.  Thus   
$g$ is linear over $\Delta$. (Using the 
effective desingularization procedure of
 \cite[2.9--2.10]{mun11}  
 we can  even insist that  $\Delta$ is   regular). 
 %
 We next  write down the set  
 $ \Delta' =  \{g(T)\subseteq \I^k\mid T\in \Delta\}$
 as the list of the set of vertices of each 
  convex polyhedron $g(T)$.
The following two conditions are now routinely checked: 
\begin{itemize}
\item[(I)]    
  $ \Delta'$  is a rational triangulation (of  $R$), and
  
  \item[(II)]     $g$ maps the set of
  vertices of $\Delta$   one-to-one into the set of vertices of
%
%
%
%
$\Delta'$. (By definition of $\Delta'$,
$g$ maps vertices of $\Delta$ onto vertices of $\Delta'$.)
  \end{itemize}
 Indeed, the separation problem has a positive answer iff the
both (I) and (II)  are satisfied: 
This completes the proof of
the decidability of the separation problem.
\end{proof}

\medskip
An  MV-algebra $D$ is  said to be 
{\it projective}
\index{Projective}
if whenever $\psi\colon A\to B$ is
a surjective homomorphism and $\phi\colon D\to B$ is a homomorphism,
there is a homomorphism $\theta\colon D\to A$ such that 
$\phi= \psi\circ \theta$.
As is well known, $D$ is projective iff it is a {\it retract}  of
a free MV-algebra $F$: in other words,
 there is a homomorphism $\omega$
of $F$ onto $D$ and a one-to-one homomorphism  $\iota$ of $D$ into
$F$ such that the composite function  $\omega\circ \iota$
is the identity function on $D$. 

\medskip

A large class of finitely generated projective
MV-algebras,  and automatically,  of finitely generated
projective  unital $\ell$-groups,  can be effectively 
presented by their generators, 
combining  the following theorem 
with the techniques of Section \ref{section:method}
below:

\begin{theorem}
\label{theorem:projective}
Let   $A$ be a finitely generated  
subalgebra of the free MV-algebra
$\McNn$. Suppose $A$ is separating
(a decidable property, by Theorem 
  \ref{theorem:separation-is-decidable}).  Then $A$ is projective.
\end{theorem}

\begin{proof}  Let  $\{g_1,\ldots,g_k\}$ 
be a generating set of  $A$. As already noted, the separation
hypothesis means that
 the   map    
 $g\colon \cube\to [0,1]^k$
defined by $g(x)=(g_1(x),\ldots,g_k(x)),$ $\,\,(x\in\cube)$ is one-to-one.
Since $g$ is continuous,
  $g$ is a homeomorphism of
$\cube$ onto its range  $R \subseteq [0,1]^k$.  
Further,   $g$ is piecewise linear and each linear piece of
$g$ is a polynomial  with integer coefficients.
Thus, $g$   is a $\mathbb Z$-map.
By  \cite[3.4]{mun11}, 
$R$ is a rational polyhedron. Further:

\begin{itemize}
\item[(i)]
The     piecewise linearity  of $g$  
yields a  triangulation  $\Delta_g$
  of $\cube\,\,$ such that   
$g$ is     linear over every simplex of
$\Delta_g$.  Since  $g$ is   a homeomorphism,
the set $g(\Delta_g)=\{g(T)\mid T\in \Delta_g\}$ is
a rational triangulation of $R$,  making $R$ into
what is known as
  an $n$-dimensional {\it PL-ball. } A classical result of
Whitehead  \cite{whi}  
shows that $R$ has a collapsible triangulation  $\nabla$.

\smallskip

\item[(ii)] 
By   \cite[4.10]{cab}, the polyhedron
 $R=g(\cube)$ has a point of denominator $1$,
  and  has a {\it strongly regular} triangulation  $\Delta$,
  in the sense that    $\Delta$ is regular and
for every maximal simplex $M$ of $\Delta$  the 
greatest common divisor  of the denominators of the
vertices of $M$ is equal to 1.

%
 \end{itemize}
 
By 
\cite[6.1(III)]{cabmun-ccm},
 the MV-algebra $\McN(R)$  is projective.
By \cite[3.6]{mun11},  $A$ is isomorphic
 to $\McN(R)$, whence the desired conclusion follows. 
\end{proof}

We refer to  \cite{mun86} and \cite{cigdotmun} for background on
 unital $\ell$-groups and their
  categorical equivalence $\Gamma$  with MV-algebras.
In particular,  \cite[4.16]{mun86} 
deals with  the freeness properties of
the unital
$\ell$-group  $\McN_{\rm group}(\cube)$ of 
all (continuous) piecewise linear
functions   $f\colon \cube\to \mathbb R$ where each linear
 piece  of $f$  has integer
coefficients, and with the constant function 1 as
the distinguished order unit.  An equivalent definition of 
$\McN_{\rm group}(\cube)$ is given by
\begin{equation}
\label{equation:gamma}
\Gamma(\McN_{\rm group}(\cube))=\McNn.
\end{equation}
Projective unital $\ell$-groups are the main concern of \cite{cabmun-ja,
cabmun-ccm}.
From the foregoing theorem we immediately have:

\begin{corollary} If  $(G,u)$ is a   finitely generated
 unital $\ell$-subgroup   of 
$\McN_{\rm group}(\cube)$, and  is
{\rm separating}, in the sense that
for each $x\neq y\in\cube$ there exists $f\in G$ such that $f(x)\neq f(y)$,
then  $(G,u)$  is projective. \hfill{$\Box$}
\end{corollary}

\section{An MV-algebraic Stone-Weierstrass theorem}
Let $P\subseteq\cube$ be a rational polyhedron and  
$\{g_1,\ldots,g_k\}$   a generating set of a subalgebra $A$
%
%
%
of $\McN(P)$. Under which conditions does $A$ coincide with 
$\McN(P)$?  

One obvious
{\it necessary} condition is that $A$ be isomorphic to
$\McN(P)$---but this condition alone is {\it not sufficient}:
for instance, 
by \cite[3.6]{mun11},  the subalgebra of
$\McN([0,1])$ generated by $x\oplus x$ is isomorphic to 
$\McN([0,1])$   but does not coincide with it,
because the points  $1/2$  and $1$ are not separated by $A,$
but are separated by the identity function  $x\in \McN(\I)$.

Another {\it necessary} condition is given by observing
that  $A$ must separate points of $P$. Again, this 
condition alone is   {\it  not sufficient} for $A$ to coincide with
$\McN(P)$, as the following example shows:

\begin{example}
\label{example:schauder} 
{\rm    Let  $A$ be the subalgebra of  
the free one-generator MV-algebra $\McN([0,1])$
generated by the two elements  $x\odot x$ and $\neg(x\oplus x)$.
See the picture below.
It is easy to see that $A$ is separating. However, $A$
does not coincide with $\McN([0,1])$: no  function  $f\in A$ satifies
$f(1/2)=1/2$, while  the McNaughton function 
$x\in \McN([0,1])$
does.}
\end{example}

\vspace{0.3cm}                                                             
\unitlength0.8cm
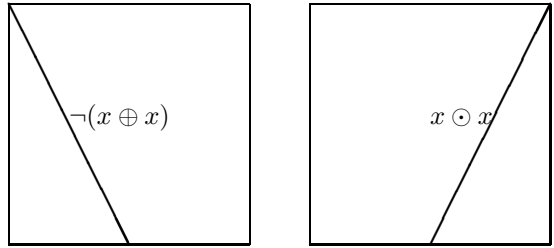
\begin{picture}(5,4)
\multiput(2.5,0)(5,0){2}{\line(1,0){4}}  
\multiput(2.5,0)(5,0){2}{\line(0,1){4}}
\multiput(2.5,4)(5,0){2}{\line(1,0){4}}
\multiput(6.5,0)(5,0){2}{\line(0,1){4}}
\thicklines
\put(2.5,4){\line(1,-2){2}} 
\put(4.5,0){\line(1,0){2}}
\put(3.5,2){$\neg (x\oplus x)$}
\put(9.5,0){\line(1,2){2}} 
\put(7.5,0){\line(1,0){2}}
\put(9.5,2){$x\odot x$}
\end{picture} 
\bigskip

\bigskip

While  individually taken, separation and isomorphism are necessary
but not sufficient conditions  for a subalgebra $A$ of 
$\McN(P)$ 
to
coincide with $\McN(P)$,  putting these two conditions together,
we will obtain in Theorem \ref{theorem:sw}
an  MV-algebraic variant of the Stone-Weierstrass theorem. 

To this purpose, let us agree to say that 
 a subalgebra $A$ of an 
MV-algebra $B$ is   an {\it epi-subalgebra} if the inclusion map is an {\it epi-homomorphism}. Stated otherwise, 
  for any two homomorphisms $h,g\colon B\to C$, if
$h\restrict A=g\restrict A$ then  
 $h=g$. We then have:

\begin{lemma}\label{Lem:SepEpi}
 Let $P\subseteq [0,1]^n$ be a rational polyhedron and $A$ a finitely generated subalgebra of $\McN(P)$. Then the following 
 conditions are equivalent:
 
\begin{itemize}
\item[(i)] $A$ is a separating subalgebra of $\McN(P)$;

\item[(ii)] $A$ is an epi-subalgebra of $\McN(P)$.
\end{itemize} 
\end{lemma}
\begin{proof}
Let $g_1,\ldots,g_n\in \McN(P)$ be a set of generators for $A$.
Then  $A$ is separating iff
the map $g=(g_1,\ldots,g_n)\colon P\to[0,1]^n$ is one-to-one.
By \cite[Theorem 3.2]{cab}, $g$ is one-to-one
iff  $g$ is a mono $\mathbb{Z}$-map.
Recalling  \eqref{equation:functor},
this latter condition is  equivalent to stating that
  the map   
 $\McN(g)\colon \McN([0,1]^n)\to \McN(P)$
  is an epi-homomorphism.
Equivalently,  the range $A$ of  
$\McN(g)$ is an epi-subalgebra of~$\McN(P)$.
\end{proof}
 
\begin{lemma}\label{Lemma_iso}
Let $P$ and $Q$ be rational polyhedra in
$\cube$  and $\eta\colon P\to Q$   a one-to-one $\mathbb Z$-map. Then the following are equivalent:
\begin{itemize}
\item[(i)] $P$ is $\,\,\mathbb Z$-homeomorphic to
$\eta(P)$;
\item[(ii)] $\eta$ is a $\,\mathbb Z$-homeomorphism 
of  $P$ onto $\eta(P)$.
\end{itemize} 
\end{lemma}
\begin{proof}
For the nontrivial direction, let 
  $\gamma\colon \eta(P)\to P$  be
   a $\mathbb Z$-ho\-m\-e\-o\-mor\-ph\-ism. 
   Then $\gamma\circ \eta$ is a one-to-one $\mathbb Z$-map from $P$ into $P$. By \cite[Theorem 3.6]{cab}, $\gamma\circ \eta$ is a $\mathbb Z$-homeomorphism. 
   It follows that
    $\eta=\gamma^{-1}\circ(\gamma\circ \eta)$ is a $\mathbb Z$-homeomorphism.
\end{proof}
We are now ready to prove the main result of this section:

\begin{theorem}
\label{theorem:sw}  Let $P\subseteq [0,1]^n$ be a rational polyhedron and $A$ a subalgebra of $\McN(P)$. Then the following
conditions are equivalent:
\begin{itemize}
\item[(i)] $A=\McN(P)$;
\item[(ii)] $A$ is isomorphic to $\McN(P)$ and 
 is a separating subalgebra of $\McN(P)$;
\item[(iii)] $A$ is isomorphic to $\McN(P)$ and  is an epi-subalgebra of $\McN(P)$.
\end{itemize} 
\end{theorem}
\begin{proof}  
(ii)$\Leftrightarrow$(iii) follows directly from Lemma
\ref{Lem:SepEpi}.
(i)$\Rightarrow$(ii) is trivial.
To prove (ii)$\Rightarrow$(i),
 let $e\colon \McN(P)\to A$ be an isomorphism,
  and $\iota \colon A\to \McN(P)$ be the inclusion map.
Since $e$ is bijective and $\iota$ is epi, 
then $\iota\circ e$ is epi. Similarly, 
 since $e$ and $\iota$ are one-to-one, 
 then so is $\iota\circ e$. 
 Let $\eta\colon P \to P $ be the unique 
 $\mathbb{Z}$-map such that 
 $\iota\circ e(f)=f\circ \eta$ for each 
 $f\in\McN(P)$. By \cite[Theorem 3.2]{cab}, $\eta$ is one-to-one and onto.  By \cite[Theorem 3.6]{cab}, 
 $\eta$ is a $\mathbb{Z}$-homeomorphism,
  that is, $\eta^{-1}$ is a $\mathbb{Z}$-map. 
Again with reference to \eqref{equation:functor}, the map 
$\McN(\eta^{-1})\colon \McN(P)\to A$ 
%
%
   is the
   inverse of $\iota\circ e$. As a consequence,
 $\iota$ is surjective and $A=\iota(\McN(P))=\McN(P)$.
\end{proof}

Recalling  \eqref{equation:gamma} we immediately
have:
\begin{corollary}
\label{corollary:sw}  A finitely generated separating
unital $\ell$-subgroup
of $\McN_{\rm group}(\cube)$ isomorphic to 
  $\McN_{\rm group}(\cube)$ coincides with  
$\McN_{\rm group}(\cube)$. \hfill{$\Box$}
\end{corollary}

When  $P=\cube$, Theorem 
\ref{theorem:sw} has the following stronger form:

\begin{theorem}
\label{theorem:sw-plus}  Any $n$-generator 
separating subalgebra   $A$ 
of $\McNn$  
 (equivalently, any $n$-generator
  epi-subalgebra of $\McNn$)
coincides with $\McNn$.
\end{theorem}

\begin{proof}    Let
 $\{g_1,\ldots,g_n\}$ be a generating
set of $A$ in $\McNn$.  The 
map  $g=(g_1,\ldots,g_n)\colon \cube\to \cube$
is  one-to-one. From  \cite[Theorem 3.6]{cab}
it follows that 
$g$ is a $\mathbb{Z}$-homeomorphism. With reference to \eqref{equation:functor},
the map 
$\McN(g)$ 
%
%
yields
 an isomorphism from $\McNn$ onto $\McNn$, whence
  $\McNn=A$. 
\end{proof}

\begin{problem} {\rm Prove or disprove}: {\it Any epi-subalgebra
$B$ of a semisimple MV-algebra C isomorphic to $C$ coincides with $C$.}
\end{problem}


\medskip
\begin{theorem}
\label{theorem:decidable-sw}
The following {\rm isomorphism problem} 
  is decidable:
  
\medskip
\noindent ${\mathsf{INSTANCE}:}$  MV-terms 
$\tau_1,\ldots,\tau_k$ in the variables $X_1,\ldots,X_n$
such that  the subalgebra  $A$ of $\McNn$ generated by
  $\hat\tau_1,\ldots,\hat\tau_k$  is  separating
  (a decidable condition, by
 Theorem  \ref{theorem:separation-is-decidable}).

\smallskip
\noindent 
${\mathsf{QUESTION}:}$  Is  $A$  isomorphic to   $\McNn$?  
  
\end{theorem}

\begin{proof}  
 Let us write
$g
=(\hat\tau_1,\ldots,\hat\tau_k)\colon \cube\to\kube$.
By hypothesis $g$ is a homeomorphism.
Let 
   $R$ be the range of $g$.
%
%
As in the proof of Theorem~\ref{theorem:separation-is-decidable},
let $\Delta'$ be a rational  triangulation of 
$[0,1]^n$  such that  $g$ is linear on each simplex of $\Delta'$.
Using the desingularization procedure of \cite[2.8]{mun11}
(which is also found in  \cite[Theorem 9.1.2]{cigdotmun}), 
we compute a  regular
subdivision  $\Delta$  of $\Delta',$ by listing the sets
of vertices of its simplexes.
We  have the following equivalent conditions:

\smallskip
$\quad \,\,\, A\cong\McNn$  

\smallskip
 $\Leftrightarrow$ $\McN(R)\cong\McNn$,
(because $A\cong \McN(R)$, by \cite[3.6]{mun11})
 
 \smallskip
 $\Leftrightarrow$ 
 $R \cong_{\mathbb Z}\cube$,  (by duality,  \cite[3.10]{mun11})
\smallskip

$\Leftrightarrow$ 
 $g$ is  a $\mathbb Z$-homeomorphism of $\cube$
 onto $R$,    
(by Lemma~\ref{Lemma_iso})

\smallskip
  $\Leftrightarrow$     
$g(S)$ is regular for each $S\in\Delta$,
\,\,and\,\, $\den(g(r))=\den(r)$ for each vertex $v$ of $\Delta$.
This last equivalence follows because
the homeomorphism
  $g$ is a  $\mathbb Z$-map
  of $\cube$  onto $R$, 
(see \cite[3.15(i)$\leftrightarrow$(iii)]{mun11}).
\end{proof}

\smallskip

\begin{proposition}
\label{proposition:partial}
The following problem is decidable:

\smallskip
\noindent
${\mathsf{INSTANCE}:}$  MV-terms  $\tau_1,\ldots,\tau_k$ in the variables
$X_1,\ldots,X_n$.

\smallskip
\noindent
${\mathsf{QUESTION}:}$ Is the subalgebra  $A$ of 
$\McNn$ generated by  $\hat\tau_1,\ldots,\hat\tau_k$
free  and  separating?
\end{proposition}

\begin{proof}
 Again, let us write
$g
=(\hat\tau_1,\ldots,\hat\tau_k)\colon \cube\to\kube$.
Let $R$ be the range of $g$.
  We first {\it claim} that
$A$ is  separating and free 
iff  $A$ is separating and 
isomorphic to $\McNn$.  

For the nontrivial direction,  as repeatedly noted, since
  $A$ is separating
and ${A\cong \McN(R)}$  then  $g$ is a homeomorphism
of   $\cube$ onto   $R$. Since  by \cite[4.18]{mun11},
 $R$ is homeomorphic to 
 the maximal spectral space  $\mu(A)$ of $A$, then $\mu(A)$
 is homeomorphic to $\cube$. 
Further, for each  $m=1,2,\ldots,$
the maximal spectral space of the free $m$-generator
 MV-algebra  $\McN(\I^m)$
is homeomorphic to $\I^m$. As is well known, whenever
 $m\not=n$ the $m$-cube
$\I^m$  is not homeomorphic to $\cube$.
Since by hypothesis $A$ is free  and finitely generated,
 the only possibility for $A$ to be isomorphic to some
free MV-algebra  $ \McN(\I^m)$
is for $m=n$, which settles our claim.

To conclude the proof, by
 Theorem \ref{theorem:sw},  $A$ is free and separating
iff  
$A$ is  (separating and) equal to  $\McNn$.
This is decidable,  by Theorems \ref{theorem:separation-is-decidable}
and   \ref{theorem:decidable-sw}. 
\end{proof}

\begin{problem}{\rm
Prove or disprove the decidability of the following problems: 

\smallskip
\noindent
(a)  
${\mathsf{INSTANCE}:}$ MV-terms  $\tau_1,\ldots,\tau_k$ in the variables
$X_1,\ldots,X_n$.

$\,{\mathsf{QUESTION}:}$  Is the subalgebra  $A$ of
$\McNn$ generated by $\hat\tau_1,\ldots,\hat\tau_k$
free?

\medskip
\noindent
(b)
${\mathsf{INSTANCE}:}$  MV-terms  $\tau_1,\ldots,\tau_k$ in the variables
$X_1,\ldots,X_n$.

$\,{\mathsf{QUESTION}:}$  Is the subalgebra  $A$ of
$\McNn$ generated by $\hat\tau_1,\ldots,\hat\tau_k$
isomorphic to $\McNn$?
}
\end{problem}

\medskip
By a quirk of fate, replacing  isomorphism by equality
in Problem (b)   we have:

\smallskip 
\begin{proposition}
\label{proposition:urto} 
The following problem is decidable:

\smallskip
\noindent
${\mathsf{INSTANCE}:}$ MV-terms  $\tau_1,\ldots,\tau_k$ in the variables
$X_1,\ldots,X_n$.

\smallskip
\noindent
${\mathsf{QUESTION}:}$  Does the subalgebra  $A$ of
$\McNn$ generated by   $\hat\tau_1,\ldots,\hat\tau_k$
coincide with     $\McNn$?
\end{proposition} 
\begin{proof}
By Theorem \ref{theorem:sw}, 
$A=\McNn$
iff  
$A$ is separating and isomorphic to $\McNn$
iff   $A$ is separating and free, by  the claim
in the proof of Proposition \ref{proposition:partial}.
The latter conjunction of properties is decidable,
by the same proposition.
\end{proof}

\bigskip

\section{Subalgebras  of  $\McNn$ and rational triangulations of $\cube$}
\label{section:method}
In this section a method is introduced to write down a list
of MV-terms  $\tau_1,\ldots,\tau_k$  in the variables
$X_1,\ldots,X_n$ in such a way that the subalgebra   
 of $\McNn$ generated by the McNaughton functions
 $\hat\tau_1,\ldots,\hat\tau_k$ is
simultaneously {\it separating and  distinct} from $\McNn$.
Conversely,  every  finitely generated separating
proper 
%
%
 subalgebra  of $\McNn$ is obtainable by this method.

In combination with 
 Theorem \ref{theorem:projective}, a large class of
projective  MV-algebras can be effectively introduced by this
method.

The procedure starts with  a  rational triangulation
$\Delta$  of $\cube$ equipped with a set
 $\mathcal H$ of  functions  $f\in \McNn$, 
 called ``hats''. Each hat of  $\mathcal H$ is 
  pyramid-shaped  and  linear
 on each simplex of $\Delta$.  One then lets $A$ be the
algebra generated by $\mathcal H$. In more detail:

\begin{definition}
\label{definition:delta-basis}
A {\it weighted triangulation} of $\cube$ 
 is a pair $(\Delta,(a_1,\ldots,a_u)),$
where  $\Delta$  is a triangulation
of $\cube$  with  rational  vertices $v_1,\ldots,v_u$
and   their {\it associated} positive integers  
$\,\, a_1,\dots,a_u, $ where for each  $i=1,\dots,u$,
 $\,\,a_i$ is   a 
 divisor of $\den(v_i)$. We write    $(\Delta,{\bf a})$
as an abbreviation of $(\Delta,(a_1,\ldots,a_u))$.

 The  function 
$h_i\colon \cube\to \I$  which is linear on
every simplex of $\Delta$  and satisfies
  $h_i(v_i)=a_i/\den(v_i)$ and
 $h_i(v_j)=0$  for each $j\not= i$
 is called the $i$th  {\it   hat}   of~$(\Delta,{\bf a})$.

Given a weighted triangulation
$(\Delta, {\bf a})$ of $\cube$, the set of  its hats 
is denoted 
 $\mathcal H_{\Delta, {\bf a}}$.
If each linear piece of
every hat  $h_i\in \mathcal H_{\Delta, {\bf a}}$
  has  integer coefficients,
  (i.e.,  $h_i\in \McNn$), we say that 
  the set  $\mathcal H_{\Delta, {\bf a}}$ is  
  {\it basic}.
\end{definition}

\begin{lemma}
\label{lemma:pugliese}
{\rm (i)}  Let  $T$ be an   $n$-simplex  
with  rational vertices $w_0,\ldots,w_n\in \cube$.
For each  $i=0,\ldots,n\,$  let $\,l_i\colon T\to \I$  be the
  linear function  
 satisfying  $l_i(w_i)=1/\den(w_i)$ and
$l_i(w_j)=0$  for  $j\not=i$.  Then $T$ is regular
iff   $\,l_i$ has integer coefficients, for each
$i=0,\dots,n$.  

\smallskip
{\rm (ii)}  Let $\Delta$ be a  regular triangulation of $\cube$ with vertices
$v_1,\ldots,v_u$.  
%
%
%
%
%
%
%
For each  $i=1,\ldots,u$ let $a_i\geq 1$ be a 
 divisor  of   $\den(v_i)$.  
Let ${\bf a}=(a_1,\ldots,a_u)$.
 Then $(\Delta, {\bf a})$ is a weighted triangulation and  
 $\mathcal H_{\Delta, {\bf a}}$ is a basic set. 
 
 \smallskip
{\rm (iii)}
 There is an effective  (=Turing-computable) procedure to test if
a  weighted triangulation $(\Delta,{\bf a})$
of $\cube$   
determines a  basic set $\mathcal H_{\Delta, {\bf a}}$.
\end{lemma}

\begin{proof} 
(i) Let $M$ be the  $(n+1)\times(n+1)$  matrix whose  $i$th
row consists of the integer coordinates of the homogeneous
correspondent  $\tilde{w_i}$ of $w_i$. 
Assume $T$ is not regular.
By definition,    $|\det(M)|\geq 2$.  The absolute value of the determinant
of the inverse matrix  $M^{-1}$   is a rational number lying
in the open interval $(0,1)$. So   $M^{-1}$ is not an integer matrix.
Since,  
 the $i$th column of  $M^{-1}$ yields the coefficients of
the linear function   $l_i\,,$  not all these functions can
have integer coefficients.
Conversely, if $T$ is regular then $M^{-1}$ is an integer matrix,
whose columns yield the coefficients of the linear functions $l_i,
\,\,\,i=1,\dots,n$.

(ii) Evidently, 
$(\Delta, {\bf a})$ is a weighted triangulation.
Fix an $n$-simplex    $T$   of
$\Delta$ with its vertices
$w_0,\ldots,w_n$ and corresponding linear functions
$l_0,\dots,l_w.$ 
By (i),  the coefficients of  each $l_i$ are integers,
and so are the coefficients  of
$a_il_i$.  Thus,   $\mathcal H_{\Delta, {\bf a}}$ 
is a basic set.

(iii)    For every $n$-simplex $T$  of $\Delta$
let $M_{T}$ be the $(n+1)\times(n+1)$
integer-valued matrices whose rows 
are the homogeneous correspondents
of the vertices  
of $T$. 
Let $D_T$ be the
$(n+1)\times(n+1)$ diagonal matrix
whose diagonal 
entries are given by the subsequence of
${\bf a}$   associated to
the vertices of $T$.
The rational matrix $M_T^{-1}D$ is effectively
computable from the input data   $(\Delta,{\bf a})$.
Arguing as in (i),  it is easy to see that
the set $\mathcal H_{\Delta, {\bf a}}$ 
is basic   iff 
$M_T^{-1}D$ is an integer matrix for each $T$.
\end{proof}

The following is an example
of a  basic set  $\mathcal H_{\Delta, {\bf a}}$
where the triangulation  $\Delta$  
is not  regular.

\begin{example} 
\label{example:continuation}
{\rm
Fix   an integer  $u\geq 3$, and let   $V=\{k/u\mid k=0,1,\ldots,u\}$.
Let $\Delta$  be the rational triangulation of $\I$
whose vertices are precisely those in $V$. Assume 
each vertex $v_k$  of $\Delta$ is associated to the integer 
$a_k=\den(v_k)$. 
We  then  have a  weighted triangulation  $(\Delta,{\bf a})$ of $\I$
and  $\mathcal H_{\Delta, {\bf a}}$ 
is a basic set.
For  $k=1,\ldots,u-1$, the
 hat $h_k$ of  $\mathcal H_{\Delta,{\bf a}}$ is a 
piecewise linear 
function with four linear pieces, connecting
the five points of the unit square  
$(0,0),((k-1)/u,0),(k/u,1),((k+1)/u,0),(1,0)$.
Each hat $h_k$ of $\mathcal H_{\Delta,{\bf a}}$ has
 value 1 at $v_k$.  In detail:
\[
h_k(x)=\begin{cases}
0&\mbox{if }\,\, 0\leq x<\frac{k-1}{u} \\
u x-(k-1)&\mbox{if }\,\, \frac{k-1}{u}\leq x<\frac{k}{u} \\  
-u x+k+1&\mbox{if }\,\, \frac{k}{u}\leq x<\frac{k+1}{u} \\  
0&\mbox{if }\,\, \frac{k+1}{u}\leq x\leq 1. \\
\end{cases}
\]
Further, 
\[
h_0(x)=\begin{cases}
-u x+1&\mbox{if }\,\, 0\leq x<\frac{1}{u} \\  
0&\mbox{if } \,\,\frac{k+1}{u}\leq x\leq 1. \\
\end{cases}\quad
h_u(x)=\begin{cases}
0&\mbox{if }\,\, 0\leq x<\frac{u-1}{u} \\
u x-(u-1)&\mbox{if }\,\, \frac{u-1}{u}\leq x\leq 1.\\  
\end{cases}
\]
A moment's reflection shows that
the  subalgebra of $\McN(\I)$ generated by 
$\mathcal H_{\Delta,{\bf a}}$ is separating and differs from
$\McN(\I)$.
}
\end{example}

\smallskip
The next two results show  that basic sets
 $\mathcal H_{\Delta, {\bf a}}$ 
generate  {\it all possible} separating
proper
 subalgebras of free MV-algebras and unital
$\ell$-groups:

\begin{theorem}
\label{theorem:delta-basis} Let $A$ be a subalgebra of $\McNn$.
\begin{itemize}
\item[(i)] $A$ is finitely generated, separating, 
and distinct from $\McNn$
iff $A$  is generated by a  basic set
$\mathcal H_{\Delta,{\bf a}}\,,$ for some
   weighted triangulation $(\Delta,{\bf a})=
   (\Delta,(a_1,\ldots,a_u))$  of  
$\cube$  such that  $a_j\not=1$ for some $j
=1,\dots,u$.

\medskip
\item[(ii)]  If
$A$ is finitely generated, separating, and distinct from $\McNn$
then for every    weighted triangulation 
$(\Delta,{\bf a})=
   (\Delta,(a_1,\ldots,a_u))$   of  $\cube$
such that  
  $\mathcal H_{\Delta,{\bf a}}$ is a basic
  generating set of  $A$, it follows that
$a_j\not=1$ for some $j=1,\dots,u$.
\end{itemize}
\end{theorem}

\begin{proof}
(i) $(\Leftarrow)$  We have only to check 
$A\not=\McNn$.  
Let  $v_1,\ldots,v_u$  be the vertices of $\Delta$.
For each  $f\in A$  the value  $f(v_j)$
is an integer multiple of   $a_j/\den(v_j)$, whence 
$f(v_j) \not=1/\den(v_j)$.
We {\it claim} that some function  in $\McNn$
  attains the value  $1/\den(v_j)$  at  $v_j$.
As a matter of fact, 
   desingularization  as in 
 \cite{mun88} or \cite[5.2]{mun11} yields
 a regular triangulation $\Sigma$  of $\I^n$ such that
$v_j$  is one of the vertices of $\Sigma$.  Let $h_j\colon \cube \to \I$
be the  Schauder hat  of $\Sigma$  at $v_j$,
as in \cite[9.1.3--9.1.5]{cigdotmun}.   By definition, 
$h_j$ is linear on every simplex of $\Sigma,\,\,$
$h_j=1/\den(v_j)$ and $h_j(v_i)=0$  for all other
vertices of $\Sigma$. Our claim is settled.

By \cite[9.1.4]{cigdotmun},
$h_j$ belongs to $\McNn$.  So 
$A$  is strictly contained in $\McNn$.

\medskip
$(\Rightarrow)$  Let  $\{g_1,\ldots,g_k\}$ be a
generating set of  $A$.
Let  $g=(g_1,\ldots,g_k)\colon \cube\to [0,1]^k$, and
$R$ be the range of $g$.  Since $A$ is separating then
$g^{-1}$ is a piecewise linear homeomorphism
of $R$ onto $\cube$  and each linear piece of $g^{-1}$
 has rational coefficients---just because each linear
 piece of $g$ has integer coefficients.
%
%
Let  $\nabla$ be a regular
triangulation of $\McN(R)$ such  that
  $g^{-1}$ is linear over every simplex of 
$\nabla$. The computability of   
 $\nabla$  follows by direct inspection of the proof 
 of  \cite[2.9]{mun11}.     
Let  $w_1,\ldots,w_u$ be the vertices of $\nabla$.
For each  $i=1,\ldots,u$ let
$v_i=g^{-1}(w_i)$.  Since $g$ has integer
coefficients  there is an integer
$1\leq a_i$  such that 
$\den(v_i)=a_i\cdot \den(w_i)$.
The set of simplexes
$$
\Delta=\{g^{-1}(T)\subseteq \cube \mid T\in\nabla\}
$$
is a rational triangulation of $\cube$.
Since  $\McNn$ strictly contains $A$, by
Theorem \ref{theorem:sw} it is impossible for
$A$ to be isomorphic to    $\McNn$.  
Since by \cite[3.6]{mun11}  $A\cong\McN(R),$ then  $\McN(R)$
is not isomorphic to $\McNn$.
By duality \cite[3.10]{mun11},
  $\cube$ is not $\mathbb Z$-homeomorphic
to $R$.  As observed in proof of Theorem~\ref{theorem:decidable-sw}, $g$ is not a
$\mathbb Z$-homeomorphism of $\cube$ onto $R$. By
\cite[3.15]{mun11} there is a rational point $r\in\cube$
such that $\den(g(r))$ is a divisor of $\den(r)$ 
different from  $\den(r)$.   
Stated otherwise,
  $\den(r)=m\cdot \den(g(r))$ with  
$m\not= 1$. 
By   \cite[5.2]{mun11}, it
is no loss of generality to assume that $\nabla$
has $g(r)$ among its vertices.  Thus, for some $j$ we can
assume that $r$ is the $j$th vertex of
$\Delta$  and write
$$
r=v_j,\,\,\,g(r)=w_j,\,\,\,\,1\not=a_j=\frac{\den(v_j)}{\den(w_j)}. 
$$
As in the proof of
 (i) above, let  $\mathcal H_{\nabla}={h_1,\ldots,h_u}$  be the
set of  Schauder hats of  $\nabla$.  
By Lemma \ref{lemma:pugliese}(i), 
for every  $i=1,\ldots,u$
each linear piece of  $h_i$
has integer coefficients. 
By   \cite[5.8]{mun11},
$\,\,\,\mathcal H_\nabla$ generates $\mathcal M(R)$.
By construction of $\nabla,$   the composite function
$h_i\circ g$ belongs to $\McNn$, has value  
$a_i/\den(v_i)\geq1/\den(v_i)$
at $v_i$, has value zero at any other vertex of $\Delta,$  and is
linear over every simplex of $\Delta$.
Therefore,  the weighted triangulation
$(\Delta,(a_1,\ldots,a_u))=(\Delta,{\bf a})$
determines the basic set    
$$
\mathcal H_{\Delta,{\bf a}}=\mathcal H_\nabla\circ g=\{h_i\circ g\mid
h_i\in \mathcal H_\nabla, \,\,\, i=1,\ldots,u \}, 
$$
which generates $A$,
just as $\mathcal H_\nabla$ generates $\mathcal M(R)\cong A$.

\medskip
(ii)  We argue by cases:

\smallskip
In case  $\Delta$  is not regular,   let
$T=\conv(w_0,\ldots,w_n)$ be an $n$-simplex in 
$\cube$ that  fails to be regular.
By hypothesis,  the linear pieces of each hat of  
$\mathcal H_{\Delta,{\bf a}}$  have  integer coefficients, 
and so do,  in particular,   the   linear functions
$l_0,\ldots,l_n\colon \mathbb R^n\to\mathbb R$
 given by the following stipulations, 
 for each  $t=0,\ldots,n:$
 \begin{itemize}
\item[---] $\,\,\,l_t$ is linear over $T$,
\item[---] $\,\,\,l_t(w_t)=a_t/\den(v_t)$, 
and 
\item[---] $\,\,\,l_t(w_s)=0$  for each $s\not=t$.
 \end{itemize}
By Lemma \ref{lemma:pugliese}(i),
 not   all $a_t$ can be
equal to 1. 

\smallskip
In case  $\Delta$ is regular, 
suppose
 $a_i=1$   for each  $i=1,\ldots,u$  (absurdum hypothesis).
 Then  $\mathcal H_{\Delta,{\bf a}}$ is precisely  the
 set of Schauder hats of  $\Delta$.   By \cite[5.8]{mun11},
$\,\,\mathcal H_{\Delta,{\bf a}}$
 generates  $\McNn$, which contradicts the hypothesis
  $A\not=\McNn$.
\end{proof}

%

\begin{corollary}
\label{corollary:delta-basis} Let $(G,u)$ be a 
unital $\ell$-subgroup of  $\McN_{\rm group}(\I^n)$.
\begin{itemize}
\item[(i)] $(G,u)$ is finitely generated, separating, 
and distinct from $\McN_{\rm group}(\I^n)$
iff 
$(G,u)$  is generated by a basic set 
$\mathcal H_{\Delta,{\bf a}}$ for some
   weighted triangulation $(\Delta,{\bf a})$  of  
$\cube$  such that   $a_j\not=1$ for some $j$.

\smallskip
\item[(ii)]  If
$(G,u)$ is finitely generated, separating, and distinct from 
$\McN_{\rm group}(\I^n)$
then for every    weighted triangulation 
$(\Delta,{\bf a})$   of  $\cube$
such that  
  $\mathcal H_{\Delta,{\bf a}}$ is a
  basic generating set of  $(G,u)$, it follows that
$a_j\not=1$ for some $j$.\hfill$\Box$
\end{itemize}
\end{corollary}

\section{Computing a basis of a 
finitely generated subalgebra of $\McNn$}
\label{section:basis}
 By    \cite[6.6]{mun11},
 every finitely generated
subalgebra $A$ of $\McNn$  is 
finitely presented,  
i.e.,  $A$ is a principal quotient of a free MV-algebra. 
Equivalently,   \cite[6.1, 6.3]{mun11},
$\,\,A$ has a {\it basis}, i.e.,  a   set of nonzero elements
$\mathcal B=\{b_1,\ldots,b_z\}$, together with
integers  
$1\leq m_1,\dots,m_z$ (called ``multipliers'')
such that
\begin{itemize}
\item[(a)]  $\mathcal B$  generates $A$.

\smallskip
\item[(b)]  
$m_1  b_1+\dots+m_z b_z=1$ 
 where the sum is computed
in the   unital
$\ell$-group  $(G,u)$  of $A$  given by $\Gamma(G,u)=A$.
See \cite[6.1(iii')]{mun11}.

\smallskip
\item[(c)] 
For each  $k$-element subset 
$C=b_{i_1},\ldots,b_{i_k}$ of $\mathcal B$  with
$b_{i_1}\wedge\ldots\wedge b_{i_k}\not=0$,
 the set of maximal ideals
of $A$ containing  $\mathcal B\setminus C$  is homeomorphic
to a $(k-1)$-simplex,  \,\,\,$(k=1,2,\ldots)$.
\end{itemize}

We now
prove that  every basic set  is a basis of the MV-algebra
it generates.

\begin{proposition}
\label{proposition:delta-basis}
Suppose the  weighted triangulation
$(\Delta,(a_1,\ldots,a_u))=(\Delta,{\bf a})$ of $\cube$ 
determines the basic set  $\mathcal H_{\Delta,{\bf a}}$.  
Let $v_1,\dots,v_u$ be the vertices of $\Delta$.
Then the  MV-subalgebra $A$ of $\McNn$ generated by
 $\mathcal H_{\Delta,{\bf a}}$   is separating,
and     $\mathcal H_{\Delta,{\bf a}}$  is a  basis of $A$%
%
%
, whose multipliers $m_i$ 
coincide with $\den(v_i)/a_i$ for each  
$i=1,\dots,u$.
\end{proposition}

\begin{proof}    $A$ is a separating subalgebra of $\McNn$
 because it is
generated by the hats of a triangulation of $\cube$.
From the  definition
 of  $\mathcal H_{\Delta,{\bf a}}$ 
 together with \cite[4.18]{mun11} and \cite[6.1(ii')]{mun11},
 it follows that
the conclusion of 
condition (c) above    
is equivalent to saying  that 
the set of points $x\in\cube$  such that
$m_{i_1}b_{i_1(x)}+\cdots+m_{i_k}b_{i_k}(x)=1$ is homeomorphic
to a $(k-1)$-simplex.    It is now easy to see that
   $\mathcal H_{\Delta,{\bf a}}$ satisfies 
    condition (c).    (See the proof of
 \cite[5.8(ii)]{mun11}).
 Condition (b) is trivially satisfied. 
Therefore, 
   $\mathcal H_{\Delta,{\bf a}}$ is a basis of $A$.
\end{proof}
%
%


\begin{remark}
By Lemma  \ref{lemma:pugliese}(iii), 
one can decide whether a weighted triangulation
$\Delta$  with multiplicities  ${\bf a}$ determines a 
basic set   $\mathcal H_{\Delta,{\bf a}}$.
By contrast, the  decidability of the problem
whether   
$\{\hat\tau_i,\ldots,\hat\tau_k\}$ is a basis
of the MV-algebra it generates is open. See
\cite[p.213]{mun11}. 
Interestingly enough, perusal of \cite[\S 6.5]{mun11} shows that
{\it every basis of $A\subseteq \McNn$ becomes
a basic generating set of $A$ after finitely many binary
algebraic blowups.}
\end{remark}

\medskip

In the light of the foregoing proposition, the following
theorem provides an effective method to transform 
every generating set of $A\subseteq \McNn$ into a 
basis
of~$A$:

\begin{theorem}
 [Effective basis generation]
\label{theorem:effective}
Every list of
terms  $\tau_1,\ldots,\tau_k$
in the variables $X_1,\ldots,X_n,$
such that  the subalgebra $A$
of $\McNn$ generated by the McNaughton functions
$\hat\tau_1,\ldots,\hat\tau_k$
is separating,  can be effectively transformed
  into a 
 weighted  triangulation  $(\Delta,{\bf a})=
 (\Delta, (a_1,\dots,a_z))$ of the $n$-cube $\cube$,
 with vertices  $v_1,\dots,v_z$,
in such a way that 
 $\mathcal H_{\Delta, {\bf a}}$ 
 is a basic generating set of  $A$.
%
%
%
%
%
%
We can effectively write down MV-terms
$\sigma_1,\ldots,\sigma_z$ representing
the hats of $\mathcal H_{\Delta, {\bf a}}$. 
\end{theorem}

\begin{proof}
Let 
$g
=(\hat\tau_1,\ldots,\hat\tau_k)
\colon \cube\to\kube$,  and
$R=$ range of $g$.  
   Since $A$ is separating,
   $g$ is a piecewise linear homeomorphism onto $R$.
The transformation proceeds as follows:

\begin{itemize}
\item[(i)] 
Arguing as in the proof of 
Theorem  \ref{theorem:separation-is-decidable},
we first compute from $\tau_1,\ldots,\tau_k$
 a regular triangulation $\Sigma$  of $\cube$
such that $g$ is linear over every simplex of $\Sigma$
(also see  \cite[18.1]{mun11}). 
$\Sigma$  is written down as the list of the
sets of  rational vertices
 of its simplexes.

\item[(ii)]  We write down  the rational
triangulation  $g(\Sigma)=\{g(T)\mid T\in \Sigma\}$ of $R$ given by the
 $g$-images of the simplexes of $\Sigma$.
 This is effective,  because
 the linear pieces  of each function  $\hat\tau_i$
 are computable from the MV-term  $\tau_i,$ and so
 is the linear piece  $l_T$  of $g\restrict T.$

\item[(iii)]   Using desingularization (\cite[2.9, 18.1]{mun11}),
we subdivide $g(\Sigma)$
into a  {\it regular}   triangulation
$\nabla$   of $R$.
Let   $\Delta$ be the  rational subdivision
 of $\Sigma$ defined by   $g(\Delta)=\nabla$. 
   Since
$g$ determines  a computable one-to-one
     correspondence   
     between the $n$-simplexes of $\nabla$
     and those of $\Delta$, then also 
     $\Delta$ is computable.
     
     \item[(iv)] 
       Let us write  $\mathcal H_\nabla$  for the
  set of Schauder hats of  $\nabla$,  \cite[5.7]{mun11}.
  For each vertex  $w$  of  $\nabla$,
  we can effectively write down an MV-term  
  $\gamma_w(X_1,\dots,X_k)$ such that
  the restriction to $R$ of the   associated
  McNaughton function  $\hat\gamma_w$
  is the hat of  $\mathcal H_\nabla$ with vertex $w$.
  Thus  $\mathcal H_\nabla$ can be effectively
  computed.
The routine verification can be made 
arguing as in  the
  proof of \cite[9.1.4]{cigdotmun}.

\item[(v)]   
  
   Observe that 
  $g^{-1}$  is  linear on every simplex of $\nabla$,
  and is   an explicitly given  piecewise linear
function. As a matter of fact, for each $n$-simplex
$U$ of  $\nabla$,  the map   
$g^{-1}\restrict U$ is an $n$-tuple of linear polynomials with
 rational coefficients that can be effectively computed from 
 the $k$-tuple of  linear polynomials with integer
 coefficients  
 %
 %
 %
%
%
%
$(\hat\tau_1,\ldots,\hat\tau_k)\restrict g^{-1}(U)$.
For each  $i=1,\dots,k$,
arguing by induction on the number of connectives
of all subterms of   $\tau_i$ one effectively
computes the linear function
$\hat\tau_i\restrict g^{-1}(U)$.

\item[(vi)]  Let  $v_1,\dots,v_z$ be the vertices of $\Delta$.
For each  $j=1,\dots,z$,  
we  set  
\begin{equation}
\label{equation:ratios}
a_j=\den(v_j)/\den(g(v_j)).
\end{equation} 
Since $g$ is piecewise linear with
integer coefficients, each  $a_j$ is an integer.
Recalling that $\circ$ denotes composition,
the  weighted triangulation 
$(\Delta,{\bf a})$ determines the basic
set    $\mathcal H_{\Delta,{\bf a}}=\{h\circ g\mid h\in \mathcal H_\nabla\}$. 
By \cite[5.8]{mun11},  $\mathcal H_\nabla$  generates  $\McN(R)$,
whence  $\mathcal H_{\Delta,{\bf a}}$ generates $A$.
\end{itemize}
%
%
%
%
%
Since   $\Delta$ is explicitly given
by listing the sets of the vertices of its simplexes,
and  the   integers  $a_1,\dots,a_z$ associated to
the vertices  $v_1,\dots,v_z$ of $\Delta$
are computed from \eqref{equation:ratios},
   the proof of \cite[9.1.4]{cigdotmun}
 yields an effective procedure to write
down  MV-terms  $\sigma_1,\ldots,\sigma_z$ such that
$\{\hat\sigma_1,\ldots,\hat\sigma_z\}= \mathcal H_{\Delta,{\bf a}}$.
\end{proof}

 \section{Recognizing  subalgebras of $\McNn$}
 \label{section:recognizing}

A main application of
 Theorem \ref{theorem:effective} is given by
 the following decidability result:  
 
 \begin{theorem}
 \label{theorem:identity} 
 The following problem is decidable:
  
 \smallskip
\noindent
${\mathsf{INSTANCE}:}$  MV-terms  $\tau_1,\ldots,\tau_k$
and $ \sigma_1,\ldots,\sigma_l$
 in the variables
$X_1,\ldots,X_n$, such that both sets
$\{\hat\tau_1,\ldots,\hat\tau_k\}$ and
$\{\hat\sigma_1,\ldots,\hat\sigma_l\}$
separate points  (the separation property being decidable by
Theorem \ref{theorem:separation-is-decidable}).

\smallskip 
\noindent
${\mathsf{QUESTION}:}$  Does  the subalgebra  $A$ of
$\McNn$ generated by   $\hat\tau_1,\ldots,\hat\tau_k$
coincide with    the subalgebra  $A'$   of
$\McNn$ generated by   $\hat\sigma_1,\ldots,\hat\sigma_l$?
 \end{theorem}
 
 \begin{proof}
 It is enough to decide $A\supseteq A'$.  By
 Theorem \ref{theorem:effective} we can safely
 suppose that for some
  weighted triangulations
 $(\Delta,{\bf a})$ and   $(\Delta',{\bf a}')$,  the MV-algebras 
 $A$ and $A'$ are respectively generated by the
 basic sets
\begin{equation}
\label{equation:indici}
\mathcal H_{\Delta,{\bf a}}=\{h_1,\ldots,h_r\}
\mbox{ \,\,\,and\,\,\,  } \mathcal H_{\Delta',{\bf a}'}=\{h'_1,\ldots,h'_s\}.
\end{equation}
   Let  
 ${\bf a}= (a_1,\ldots,a_r),\,\,\,{\bf a}'=(a'_1,\ldots,a'_s)$. 
 Let $h=(h_1,\ldots,h_r)\colon \cube\to \kube$
 and $R=h(\cube)$ be the range of $h$.
The separation hypothesis is to the effect 
that  $h$ is a piecewise linear
 homeomorphism of $\cube$ onto $R$.
 For each  $i=1,\dots,r,$  all linear pieces
 of $h_i$  have integer coefficients.
As in the proof of Theorem 
\ref{theorem:separation-is-decidable}
(also see \cite[18.1]{mun11}),  we now compute a rational
subdivision $\Delta^*$  of 
$\Delta'$  such that  
$h$ is linear on each simplex of 
$\Delta'$  (by definition of  $\mathcal H_{\Delta,{\bf a}}$\,,
\,\,$\Delta^*$ is automatically
a subdivision of $\Delta$). 
Then the set
$$
h(\Delta^*)=\{h(T)\mid T\in \Delta^*\}
$$
is a rational triangulation of $R$.
%
%
%
%
Let   $\nabla$ be 
the
regular subdivision of 
$h(\Delta^*)$ obtained by the desingularization
process in  \cite[2.9]{mun11}. Let 
$w_1,\ldots,w_u,$ be the vertices of $\nabla$ and  $p_1,\ldots,p_u$ 
their respective
Schauder hats. 
Since all steps of the desingularization
process in  \cite[2.9]{mun11} 
are effective,  $\nabla$  is effectively computable.
Upon writing
$$
\Sigma=h^{-1}(\nabla)=\{h^{-1}(U)\mid U\in \nabla\},
$$
we get  a rational triangulation $\Sigma$ of $\cube$
that jointly subdivides $\Delta$ and $\Delta',$ 
%
%
(whence  $h$ is linear on
each simplex of $\Sigma$).
   Evidently,  $\Sigma$  is computable as the list
   of the sets of  vertices of its simplexes.
Let $v_1=h^{-1}(w_1),\ldots,v_u=h^{-1}(w_u)$ be the vertices of $\Sigma$.
The piecewise linear functions 
$$
q_1=p_1\circ h,\ldots,q_u=p_u\circ h
$$
have integer coefficients.
%
%
%
Since $\den(h(v_i))$ is a divisor
of  $\den(v_i)$, the rational number  $q_i(v_i)=1/\den(h(v_i))$ is an
integer multiple of $1/\den(v_i),$  say
 $$\mathbb Z \ni c_i=\den(v_i)/\den(h(v_i)),\,\,\, (i=1,\dots,u).$$ 
Letting now 
${\bf c}=(c_1,\ldots,c_u)$,    we have a 
weighted triangulation  $(\Sigma,{\bf c})$ 
such that  $\mathcal H_{\Sigma,{\bf c}}$ is a basic generating set
of  $A$, just as  the set
$\{p_1,\ldots,p_u\}$  generates $\McN(R)$, 
by \cite[5.8]{mun11}.

\smallskip
Recalling \eqref{equation:indici},
we are now  ready to decide whether  $A\supseteq A'$ as follows:
\smallskip

For every  $j=1,\ldots,s$  and  $n$-simplex  $T\in \Sigma$   let
  $f_{T,j}=h'_j\restrict T$ be the restriction to $T$ of the $j$th 
hat of the basic set  $\mathcal H_{\Delta',{\bf a}'}$. 
  Observe that $f_{T,j}$ is linear
on $T$ and has linear coefficients. We now check whether  $f_{T,j}$
is obtainable as a sum of   positive  ($>0$)    integer multiples  
 of some of the  hats $q_i\restrict T$.
Let  $V_T$ be the set of vertices of $\Sigma$
lying in $T$, with the corresponding set of hats  $H_T=\{q_v\mid v\in T\}
\subseteq \{q_1,\ldots,q_u\}$=the hats of the basic set 
$\mathcal H_{\Sigma,{\bf c}}$ of $A$.

\medskip
\noindent {\it Case 1:}    For each  $j=1,\ldots,s,$    $\,\,\,T\in \Sigma$ and
  $v\in V_T$,   $q_{v}(v)$ is a divisor of
$f_{T,j}(v)$.

Then the linear function  $f_{T,j}$ coincides (over $V_T$, 
and hence) 
over $T$
with a suitable sum of positive integer multiples of  the hats of
$H_T$.   Direct inspection shows that
 $h'_j$ is a sum of integer multiples of some hats in 
 $\mathcal H_{\Sigma,{\bf c}}$.
 It follows that
  $\mathcal H_{\Sigma,{\bf c}}$
 generates $A',$  and we conclude   $A\supseteq A'$.

\medskip
\noindent {\it Case 2:}   For some  $j=1,\ldots,s,$    $\,\,\,T\in \Sigma$ and
  $v\in V_T$,
  $q_{v}(v)$ is not a divisor of
$f_{T,j}(v)$. 

The possible values of functions in $A$ at $v$
are integer multiples of  $q_{v}(v)$, because all  other  hats
of $\mathcal H_{\Sigma,{\bf c}}$ 
vanish at $v$.
Therefore,   $h_j$  does not belong to $A$, whence
the inclusion   $A\supseteq A'$  fails.

\medskip This completes the decision procedure
for   $A= A'$.
 \end{proof}
 
\smallskip
\begin{problem}
{\rm 
Prove or disprove the decidability of the following problem:

\smallskip
\noindent
${\mathsf{INSTANCE}:}$  MV-terms  $\tau_1,\ldots,\tau_k$ and
 $\sigma_1,\ldots,\sigma_l$ in the variables
$X_1,\ldots,X_n$.

\smallskip
\noindent
${\mathsf{QUESTION}:}$  Is the subalgebra  of
$\McNn$ generated by   $\hat\tau_1,\ldots,\hat\tau_k$
equal to    the subalgebra  of
$\McNn$ generated by   $\hat\sigma_1,\ldots,\hat\sigma_l$?
}
\end{problem}

\section{Conclusions: Two types of  presentations}
\label{section:final}

Throughout  this paper, MV-algebras
$A\subseteq \McNn$  (resp., unital $\ell$-groups $(G,u)\subseteq
\McN_{\rm group}(\I^n)$\,)
 have been ``effectively presented'' 
 by  
 a finite string of symbols
$\tau_1,\ldots,\tau_u$,  where each  $\tau_i=
\tau_i(X_1,\ldots,X_n)$ is an MV-term
(resp., a unital $\ell$-group term) 
in the variables $X_1,\ldots,X_n$.
%
%
%
%
 
 Traditional  finite presentations   are instead
defined as we did in Section \ref{section:presentation},   
  by a single  $k$-variable term  $\sigma$, letting
 $Z_\sigma\subseteq \I^k$  be the zeroset  
 $\hat\sigma^{-1}(0)$ of the
 McNaughton function 
 $\hat\sigma\colon \I^k\to \I$ associated to $\sigma$, and  
 setting  
 \begin{equation}
 \label{equation:finale}
 A_\sigma=\McN(\I^k)/ \mathfrak j_\sigma \cong \McN(Z_\sigma),
 \end{equation}
  where  $\mathfrak j_\sigma$ is the principal ideal of 
 $\McN(\I^k)$  generated by  $\hat\sigma$. 
 One similarly defines  finite presentations of unital $\ell$-groups,
 \cite{cab,cabmun-ccm, mun08}.

By \cite[6.6]{mun11},
every finitely generated subalgebra of $\McNn$
is finitely presented, but  not every finitely presented
MV-algebra   is isomorphic to  
a subalgebra of  a free MV-algebra.  
For instance,   
$\{0,1\}\times\{0,1\}$
(and more generally, every non-simple finite MV-algebra) is finitely
presented but is not isomorphic to a subalgebra
of a free MV-algebra, because its maximal spectral
space is disconnected.  
Thus one may reasonably  expect that decision problems that
are unsolvable for finitely presented
 $\ell$-groups and open for finitely presented unital
$\ell$-groups  (equivalently,  for finitely presented MV-algebras),
turn out to be decidable for {\it separating} 
subalgebras
of $\McNn$
presented  via their generators $\hat\tau_1,\ldots,\hat\tau_u$. 
Here is an example of this state of affairs:

\begin{itemize}
\item[---] As shown in \cite[Theorem D]{glamad},
 for each fixed $k\geq 6$
the property of being a free $k$-generator
$\ell$-group is undecidable.  This follows from 
Markov's  celebrated unrecognizability theorems  (see \cite{sht}
for a detailed account.)

\item[---]  The same problem for unital $\ell$-groups and MV-algebras
 is open, except for  $k=1$, where the
 problem is decidable,  (see \cite[18.3]{mun11}).
 
\item[---]   Theorem \ref{theorem:decidable-sw}
shows the decidability of 
the  problem whether
a  finitely generated {\it separating} subalgebra of $\McNn$
is isomorphic to $\McNn$.
\end{itemize}

The  separation hypothesis   plays a crucial role in
most decidability results of the earlier sections.
As a matter of fact, 
the final two results of this paper will show that
(without the separation hypothesis),  
for all decision
problems concerning
  finitely generated subalgebras $A$ of free algebras,
it is immaterial whether $A$ is presented by a list of
generators
$\hat\tau_1,\ldots,\hat\tau_u$ 
 or by a principal ideal $\mathfrak j_\sigma$  
of some free algebra.

\begin{theorem}
\label{theorem:generators-to-quotient}
There is a computable transformation
of every presentation of an MV-algebra
 $A\subseteq \McNn$ by a list of generators
$\hat\tau_1,\ldots,\hat\tau_u$,  into a presentation
of an isomorphic copy  $A_\sigma$ of 
$A$ as a principal quotient of some finitely
generated free MV-algebra as in \eqref{equation:finale}.  
\end{theorem}

\begin{proof}
Following  \cite[18.1]{mun11}, from the input
MV-terms $\tau_1,\ldots,\tau_u$ 
 we first compute the rational polyhedron  $R\subseteq \I^u$ 
 given by  the range of the function  $g=(\hat\tau_1,\ldots,
\hat\tau_u)$.  
By \cite[3.6]{mun11},    $A\cong \McN(R)$.
Next, in the light of \cite[2.9, 18.1]{mun11},
we list the sets of vertices of the simplexes of a regular  
triangulation $\Delta$  of $\kube$ such  
that  the set  $\Delta_R=\{T\in \Delta\mid T\subseteq R\}$ is a 
 triangulation of $R$.
 Without loss of generality,  $\Delta_R$ is {\it full}: any simplex
 of $\Delta$ all of whose vertices lie in $\Delta_R$ is a simplex
 of $\Delta_R$.  
Following  the proof of
 \cite[9.1.4(ii)]{cigdotmun}, we
compute  MV-terms $\rho_1,\dots, \rho_w$ in the variables
$Y_1,\dots,Y_u,$
whose associated McNaughton functions  $\hat\rho_1,\dots, 
\hat\rho_w$ constitute
the set   $\mathcal H_\Delta$  
 of Schauder hats of $\Delta$, as defined in \cite[9.1.3]{cigdotmun}.
 The $\oplus$-sum 
  of all hats  with vertices not belonging to $R$
   (coincides with their pointwise sum taken in the
   unital $\ell$-group  $\McN_{\rm group}(R)$ and)
 provides an  MV-term
 $\sigma(Y_1,\ldots,Y_u)$,  together with 
 its associated  McNaughton function $\hat\sigma\in \McN(\I^u)$.
 Since $\Delta_R $ is full,  the
  zeroset $Z_\sigma$ of  $\hat\sigma$  coincides with $R$.  
The  isomorphisms
   $$A\cong\McN(g(\I^n)) = \McN(R)=\McN(Z_{\sigma})\cong 
     \McN(\I^u)/\mathfrak j_\sigma = A_\sigma, $$
yield a finite presentation of $A$ as a principal
quotient of $ \McN(\I^u)$.
\end{proof}

Conversely, we can prove:

\begin{theorem} 
\label{theorem:finale} 
For any
arbitrary input MV-term  $\sigma=\sigma(Y_1,\ldots,Y_k)$, we have:

\begin{itemize}
\item[(i)] It is decidable whether the MV-algebra 
 $A_\sigma=\McN(\I^k)/ \mathfrak j_\sigma\cong
 \McN(Z_\sigma)$ 
is isomorphic to a 
subalgebra  of a free MV-algebra.

\smallskip
\item[(ii)]
In case $A_\sigma$
is isomorphic to a 
subalgebra  of a free MV-algebra,   
$\sigma$ can be effectively transformed  into a 
finite  list of MV-terms
$\tau_i$
in $n$  variables,  in such a way that
$A_\sigma$  is isomorphic to 
the 
subalgebra $A$ of the free MV-algebra $\McNn$
generated by the set of   
$\hat\tau_i$.
\end{itemize}
\end{theorem}

\begin{proof}
Using a variant of  the algorithm $\mathsf{Mod}$ of
 \cite[18.1]{mun11},
 we first  compute
 a
 regular triangulation  $\Delta$
 whose support is the zeroset $Z_\sigma \subseteq \kube$ 
of $\hat\sigma$.
The proof of (i) and (ii) then proceeds as follows:

\smallskip
(i)  In \cite[4.10]{cab} it is proved that
$A_\sigma$ is isomorphic to a (necessarily finitely generated) 
subalgebra  of a free MV-algebra iff the
following three conditions hold:

 \begin{itemize}
\item[(a)]  $Z_\sigma$ intersects the set of vertices of $\kube$;  
\item[(b)]  $Z_\sigma$ is connected; 
\item[(c)] $Z_\sigma$ is strongly regular.  
\end{itemize}

From  $\Delta$, explicitly given by the sets of vertices
of its simplexes,   conditions (a)-(b)
can be  immediately checked.
To check condition (c), for each
maximal simplex $M$ of $\Delta$, one checks whether the greatest common divisor of the denominators
of the vertices of $M$ equals 1. This completes the proof of (i).

\medskip


(ii) Assuming 
 all three checks (a)-(c)  are successful, following
 the proof of the characterization theorem 
 \cite[4.10]{cab}  we output the desired
MV-terms  $\tau_i$ through  the following steps: 
 
 \begin{itemize} 
 \item[(d)]  
From $\Delta$  we compute a  
{\it collapsible} triangulation $\Delta'$ whose support
  lies in the cube  $\cube$,  for some suitably large
  integer  $n$.

    \item[(e)] Next we compute  a simplicial map from 
    $\Delta'$ to $\Delta$, providing
    a  $\mathbb Z$-map 
    $\eta$ from the support of $\Delta'$ onto the support
    $Z_\sigma $ of $\Delta$.

    \item[(f)] 
Following now  the proof of  \cite[5.1]{cabmun-ccm},
we compute a
    $\mathbb Z$-map  $\gamma$  which is a retraction
    of $\cube$ onto the support of $\Delta'$.

   \item[(g)] 
Letting  $\pi_i\colon \I^k \to \I$ denote the $i$th
 coordinate function, we observe that   for each
 $i\in\{1,\ldots,k\}$  the $\mathbb Z$-map
 $
 \pi_i\circ \eta\circ\gamma\mbox{ belongs to } \McN(\cube).
 $
Since both $\mathbb Z$-maps $\gamma$ and $\eta$
 are explicitly given,
 an application of \cite[9.1.5]{cigdotmun} yields
 MV-terms   $\tau_1,\ldots,\tau_k$ in the
 variables $X_1,\dots,X_n$ such that 
  $\hat\tau_i=\pi_i\circ\gamma\circ \eta$.
  \end{itemize}

  Let $A$ be the subalgebra of
  $\McNn$  generated by
  $\hat\tau_1,\dots, \hat\tau_k$.
The final part of the proof of  \cite[4.10]{cab} yields  
     \[A_\sigma=
     \McN(\I^u)/\mathfrak j_\sigma \cong
     \McN(Z_{\sigma})=
     \McN((\gamma\circ\eta)(\I^n))\cong
     A.\hfill\qedhere
     \] 
\end{proof}

\section*{Funding}
The first author was supported by a 
 Marie Curie Intra European Fellowship 
 within the 7th European Community Framework 
 Program (ref. 299401, FP7-PEOPLE-2011-IEF).



{\small
\begin{thebibliography}{999}

%
%
%

\bibitem{bey}
W. M.  Beynon,
Applications of duality in the theory of finitely 
generated lattice-ordered abelian groups,  Canad. J. Math., 29(2):243--254, 1977.
 

\bibitem{bigkeiwol} {A.
     Bigard, K. Keimel,  S. Wolfenstein,}
{Grou\-pes et Anneaux
     R\'{e}ticul\'{e}s},   Lecture Notes in
Mathematics, Springer-Verlag, Berlin, volume 608, 1971.



\bibitem{boc}
{  F.  Boca}, 
An AF algebra associated with the Farey tessellation,
{  Canad. J. Math}.,    
{ 60}:975--1000,   2008.


\bibitem{cab}
L. Cabrer, Simplicial geometry of unital
lattice-ordered abelian groups,
Forum Mathematicum (in press, DOI: 10.1515/forum-2011-0131, 
arXiv:1202.5947).

%
  
    \bibitem{cabmun-ja}
    L.  Cabrer,  D. Mundici, 
    Finitely presented lattice-ordered abelian groups with order unit,
J. Algebra, 343(1):1--10,  2011.
 
 
 
 
  \bibitem{cabmun-ccm}
L. Cabrer,   D. Mundici, 
Rational polyhedra and projective lattice-ordered
abelian groups with order unit, 
Communications in Contemporary Mathematics,  
14(3), 2012   (DOI: 10.1142/S0219199712500174). 





\bibitem{cigmar}
R. Cignoli, V. Marra,
Stone duality for real-valued multisets,
Forum Mathematicum,  24(6):1317--1331,  2012.




\bibitem{cigdotmun} { R. Cignoli,  I.M.L.
D'Ottaviano, D. Mundici},
{Algebraic Foundations of many-valued
Reasoning}, Trends in Logic,
vol.  7, Kluwer Academic Publishers, Dordrecht, 2000.

%

\bibitem{eck}
C. Eckhardt,
A noncommutative Gauss map, 
Math. Scand.,
108:233--250, 2011.

 

%
%
%
%

\bibitem{forum}
M. Fedel, K. Keimel, F. Montagna,  W. Roth,
Imprecise probabilities, bets and
functional analytic methods in \L ukasiewicz logic,
Forum Mathematicum, 25(2): 405--441, 2013.
 

 

\bibitem{gla} 
A. M. W. Glass, Partially Ordered Groups,
Series in Algebra,
Vol.7, World Scientific, Singapore, 1999.
 



\bibitem{glamad}
{  A.~M.~W. Glass,  J.~J.  Madden,}
    The word problem versus the 
isomorphism problem,
 J. London Math. Soc., 
30(2):53--61, 1984.

 
 \bibitem{jer}
 E. Je\v{r}\'abek,
 The complexity of admissible rules in \luk\ logic,
 Journal of Logic and Computation,
23(3):693--705, 2013.
 

%
%

\bibitem{mmm}
C. Manara, V. Marra, D. Mundici, 
Lattice-ordered Abelian groups and
Schauder bases of unimodular fans, Transactions of the American Mathematical
Society, 359:1593--1604,  2007. 

%
%
%


\bibitem{marspa}
V. Marra, L. Spada, Duality, projectivity, and unification
in \luk\ logic and MV-algebras, Annals of Pure and
Applied Logic, 164(3):192--210, 2013.

 

\bibitem{mun86}
{    D.  Mundici}, Interpretation of    
AF $C^*$-algebras in \L ukasiewicz 
sentential calculus, { 
J. Functional  Analysis}   
{ 65}:15--63,  1986.


\bibitem{mun88}
{    D.  Mundici},  Farey stellar subdivisions,
ultrasimplicial groups, and $K_0$ of AF algebras,
{  Advances in Mathematics},  { 68}:23--39, 1988.

 
%
%
%
%
   
   
\bibitem{mun08}
{D.  Mundici},
The Haar theorem for lattice-ordered
abelian groups with order  unit,
{  Discrete and continuous dynamical systems},
{ 21}:537--549,  2008.



\bibitem{mun-lincei}
 { D.  Mundici},
 Recognizing the Farey-Stern-Brocot AF algebra, 
Dedicated to the memory of Renato Caccioppoli.
{ Rendiconti Lincei Mat. Appl.,}
20:327--338,  2009.

 
%
%

\bibitem{mun11}
D. Mundici, Advanced \L ukasiewicz calculus and MV-algebras,
Trends in Logic, Vol. 35, Springer-Verlag, Berlin, NY, 2011.


\bibitem{mun-cpc}
D. Mundici, 
Invariant measure under  the  affine group over $\mathbb Z$,
Combinatorics, Probability and Computing, (in press,  arXiv  1102.0897)

 

\bibitem{ped}
A. Pedrini, 
The Euler characteristic of a polyhedron as a valuation on its
coordinate vector lattice.
Preprint arXiv: 1209.3248

%

\bibitem{sht}
M. A.  Shtan'ko,
Markov's theorem and algorithmically
non-recognizable combinatorial manifolds,
Izvestiya Mathematics, 68.1:205--221,  2004.

%

\bibitem{sta}
 J. R. Stallings,
  Lectures on Polyhedral Topology,
Tata Institute of Fundamental Research,
Mumbay, 1967.


\bibitem{whi}
J.  H.  C.  Whitehead, 
Simplicial spaces, nuclei and m-groups, 
Proc. Lond. Math. Soc., 45:
243--327, 1939.

\end{thebibliography}
 }
 \end{document}